\documentclass{amsart}
\usepackage{amssymb}
\usepackage{graphicx}
\usepackage[all]{xy}
\usepackage{a4wide}
\usepackage{hyperref}
\newcommand{\bbZ}{\mathbb{Z}}
\newcommand{\bft}{\mathbf{t}}
\newcommand{\bfa}{\mathbf{a}}
\newcommand{\bfsigma}{{\mbox{\boldmath $\sigma$}}}
\newcommand{\gal}{{\rm Gal}}
\newcommand{\nek}{,\ldots,}
\newcommand{\afrak}{\mathfrak{a}}

\newcommand{\Tr}{{\rm Tr}}

\renewcommand{\phi}{\varphi}

\newtheorem{lemma}{Lemma}[section]
\newtheorem{proposition}[lemma]{Proposition}

\newtheorem{theorem}{Theorem}
\newtheorem*{thm}{Theorem}

\theoremstyle{remark}
\newtheorem{remark}[lemma]{Remark}

\newtheorem{question}{Question}
\begin{document}

\fontsize{12}{18}\selectfont

\CompileMatrices 

\title{On PAC extensions and scaled Trace forms}%
\author{Lior Bary-Soroker}%
\address{
The Raymond and Beverly Sackler
School of Mathematical Sciences
Tel Aviv University
Ramat Aviv, Tel Aviv 69978
ISRAEL}%
\email{barylior@post.tau.ac.il}%
\author{Dubi Kelmer}%
\address{
The Raymond and Beverly Sackler
School of Mathematical Sciences
Tel Aviv University
Ramat Aviv, Tel Aviv 69978
ISRAEL}%
\email{kelmerdu@post.tau.ac.il}%
%
%
\begin{abstract}
Any non-degenerate quadratic form over a Hilbertian field (e.g., a
number field) is isomorphic to a scaled trace form. In this work we
extend this result to more general fields. In particular,
prosolvable and prime-to-$p$ extensions of a Hilbertian field. The
proofs are based on the theory of PAC extensions.
\end{abstract}
\maketitle
\section{Introduction and results}

Let $F/K$ be a finite separable field extension. It is equipped with
a \textit{trace form}, $x \mapsto \Tr_{F/K} (x^2)$. The
characterization of trace forms has been initiated by Conner and
Perlis, who were interested in the following question: Which
quadratic forms over $\mathbb Q$ are Witt-equivalent to a trace
form. In \cite{ConnerPerlis1984}, they showed that these forms are
precisely the positive non-degenerate quadratic forms (where
positive means signature $\geq 0$). This result was generalized to
number fields \cite{Scharlau1987}, and to some Hilbertian fields
\cite{KruskemperScharlau1988}. In \cite{Epkenhans1993}, Epkenhans
showed that over a number field any non-degenerate positive
quadratic form of dimension $\geq 4$ is isomorphic to a trace form,
and classified all classes of trace form of dimension $\leq 3$,
completing the work of Kr\"uskemper \cite{Kruskemper1992}.

There is a natural generalization of the trace form, that is, the
\textit{scaled trace form}. Namely, any separable extension $F/K$
and any nonzero element $\alpha\in F$ admits the non-degenerate
quadratic form
\[
x \mapsto \Tr_{F/K}
(\alpha x^2), \qquad x\in F.
\]
A natural question is the following. Given a field $K$, can any
non-degenerate quadratic form over $K$ be realized as a scaled trace
form.

In \cite{Scharlau1987,Waterhouse1986}, Scharlau and Waterhouse,
independently, gave a positive answer for number fields or more
generally for Hilbertian fields of characteristic $\neq 2$. Recall
that a Hilbertian field $K$ is a field with the property that for
any irreducible polynomial $f(T,X)$ over $K(T)$ there exist
infinitely many $a\in K$ for which $f(a,X)$ is irreducible.
\begin{thm}[Scharlau-Waterhouse]
Any non-degenerate quadratic form over a Hilbertian field of
characteristic $\neq 2$ is isomorphic to a scaled trace form.
\end{thm}


In this note, we generalize this result to a larger class of fields.
Note that an obvious necessary condition for a quadratic form of
dimension $n$ over $K$ to be isomorphic to a scaled trace form is
that $K$ has a separable extension of degree $n$. Thus, a more
subtle question is the following:
\begin{question}\label{que:main}
Given a field $K$ having a separable extension of degree $n$, can
any non-degenerate quadratic form of dimension $n$ over $K$ be
realized as a scaled trace form.
\end{question}
\begin{remark}
In general the answer to this question is negative. For example,
$\mathbb{R}$ has a unique separable extension, $\mathbb{C}$, of
degree $2$. Any scaled trace form $x\mapsto \Tr_{\mathbb C/\mathbb
R}(\alpha x^2)$ is isotropic (since for $x=\sqrt{i/\alpha}$ we have
$\Tr_{\mathbb C/\mathbb R}(\alpha x^2)=\Tr_{\mathbb C/\mathbb
R}(i)=0$). Hence, the non-isotropic quadratic form $\langle
1,1\rangle$ is not isomorphic to a scaled form.
\end{remark}

\subsection*{Prosolvable extensions} An extension
$K/K_0$ is called prosolvable, if any finite subextension $L_0/K_0$
(with $L_0\subseteq K$) is solvable, i.e., the Galois group of the
Galois closure of $L_0/K_0$ is solvable.

\begin{theorem}\label{thm:sol}
Let $K$ be a prosolvable extension of a Hilbertian field of
characteristic $\neq 2$. Then every non-degenerate quadratic form
over $K$ of dimension $>4$ is isomorphic to a scaled trace form.
\end{theorem}

Note that there are solvable extensions having no separable
extensions of degree $\leq 4$ (e.g., the maximal solvable
extension), so in this generality the condition on the dimension is
necessary. For prosolvable extension and quadratic forms of
dimension $n=3,4$, we do not know the answer for Question
\ref{que:main}. However in case $n=2$ we show that the answer is
negative (Proposition~\ref{prop:prosolvable}).

\subsection*{Prime-to-$p$ extensions} Let $p$ be a prime number.
An algebraic extension $K/K_0$ is called prime-to-$p$, if $p$ does
not divide the degree of any finite subextension $L_0/K_0$ (with
$L_0\subseteq K$).

\begin{theorem}\label{thm:pi}
Let $p$ be a prime and let $K$ be a separable extension of a
Hilbertian field $K_0$ of characteristic $\neq 2$ such that the
Galois closure of $K/K_0$ is prime-to-$p$ over $K_0$. Then every
non-degenerate quadratic form over $K$ is isomorphic to a scaled
trace form.
\end{theorem}

Notice that for $p=2$ and quadratic forms of dimension $>4$,
Theorem~\ref{thm:pi} is a special case of Theorem~\ref{thm:sol}.
(Recall that, by the famous Feit-Thompson theorem, any group of odd
order is solvable.)
\subsection*{PAC fields}
Hilbert's Nullstellensatz asserts that every variety defined over an
algebraically closed field $K$ has a ($K$-rational) point. A field
satisfying this property is called Pseudo Algebraically Closed
(abbreviated PAC). An equivalent condition for a field $K$ to be PAC
is that any absolutely irreducible polynomial in two variables
$f(X,Y)$ with coefficients in $K$ has infinitely many roots
$(x,y)\in K^2$ \cite[Theorem~11.2.3]{FriedJarden2005}. There are
abundance of PAC fields, in fact, in some sense most algebraic
extensions of a countable Hilbertian field with a finitely generated
absolute Galois group are PAC \cite[Theorem.
18.6.1]{FriedJarden2005}. Over PAC fields we give a positive answer
to Question~\ref{que:main}.

\begin{theorem}\label{thm:1}
Let $K$ be a PAC field of characteristic $\neq 2$. A non-degenerate
quadratic form of dimension $n$ over $K$ is isomorphic to a scaled
trace form if and only if $K$ has a separable extension of degree
$n$.
\end{theorem}

\subsection*{PAC extensions}
In \cite{JardenRazon1994}, Jarden and Razon generalized the notion
of PAC fields to field extensions: A field extension $M/K$ is said
to be a \textbf{PAC extension} if for any absolutely irreducible
polynomial in two variables $f(X,Y)\in M[X,Y]$, separable in $Y$
(i.e., $\frac{\partial f}{\partial X}\neq 0$) there are infinitely
many $(x,y)\in K \times M$ such that $f(x,y)=0$.
Clearly, a field $K$ is PAC (as a field) if and only if the trivial
extension $K/K$ is PAC (as an extension). On the other hand, we note
that a field extension $M/K$ with $M$ a PAC field is not necessarily
a PAC extension, see \cite[Page 9]{Bary-SorokerDTRF}. For fields
having a PAC extension we give a partial answer for Question
\ref{que:main}.

\begin{theorem}\label{thm:main}
Let $K$ be a field of characteristic $\neq 2$. Assume that $K$ has a
PAC extension $M/K$ which has a separable extension of degree $n$.
Then every non-degenerate quadratic form of dimension $n$ over $K$
is isomorphic to a scaled trace form.
\end{theorem}

This theorem is the main theorem of the paper and we deduce all the
previous theorems from it. For this deduction, we use the result of
Jarden and Razon showing that, as for PAC fields, most algebraic
extensions of a countable Hilbertian field with a finitely generated
absolute Galois group are PAC extensions \cite[Proposition
3.1]{JardenRazon1994}.

%

\subsection*{Outline}
In section  \ref{sec:2} we prove Theorem~\ref{thm:main}. The proof
goes along the lines of Scharlau-Waterhouse, where the use of
Hilbertianity is replaced by a weaker property (Lemma
\ref{prop:weakHIT}). In order to apply this property, we have to
compute the Galois group of the characteristic polynomial of a
generic symmetric matrix times some diagonal matrix.
Theorem~\ref{thm:1} is then an obvious result of
Theorem~\ref{thm:main} (by taking $M=K$). In section \ref{sec:3} we
use the abundance of PAC extensions over a Hilbertian field to
deduce Theorems \ref{thm:sol} and \ref{thm:pi} from Theorem
\ref{thm:main}. In what follows all fields are assumed to have
characteristic $\neq 2$.

\vspace{10pt}

\noindent\textsc{Acknowledgments.} We thank Ido Efrat, Dan Haran,
Moshe Jarden, and Ze\'ev Rudnick for helpful comments on an earlier
draft of this paper. This work is part of the first author's Ph.D.
thesis done at Tel Aviv University, supervised by Prof. Dan Haran.

\section{Proof of Theorem~\ref{thm:main}}\label{sec:2}
We start with a few auxiliary lemmata. The first lemma is a known
result from the theory of Hermitian forms, giving a condition for a
quadratic form to be isomorphic to a scaled trace form.

\begin{lemma}\label{lem:Her}
A nonzero symmetric matrix $D\in {\rm Mat}_n(K)$ represents a scaled
trace form over $K$ if and only if there is another symmetric matrix
$A\in {\rm Mat}_n(K)$ such that the characteristic polynomial of
$AD$ is separable and irreducible over $K$.
\end{lemma}

\begin{proof}
See e.g. \cite{Scharlau1987,Waterhouse1986}.
\end{proof}

Let $D = {\rm diag}( d_1 ,\ldots,d_n)$ be a diagonal matrix with
$d_i\neq 0$ over some field. Let $T = (t_{ij})$ be a generic
$n\times n$ symmetric matrix, i.e., the only algebraic relations are
$t_{ij}=t_{ji}$. Let $p(\bft, x)$ be the characteristic polynomial
of $TD$, that is $p(\bft,x) = \det(xI - TD)$.

\begin{lemma}\label{lem:GaloisGroup}
Let $L$ be a field such that $d_i\in L$, $i=1\nek n$. Then
$p(\bft,x)$ is a separable irreducible polynomial over $L(\bft)$ and
the Galois group of $p(\bft,x)$ over $L(\bft)$ is the symmetric
group $S_n$.
\end{lemma}

\begin{proof}
The assertion is trivial for $n=1$. For $n=2$ the polynomial is
separable and irreducible (and the Galois group is $S_2\cong \mathbb
Z/2\mathbb Z$) if and only if the discriminant $\Delta$ is not in
the field $L(\bft)$. We can write explicitly $\Delta^2 =
d_1^2t_{11}^2 +d_2^2t_{22}^2+4d_1d_2 t_{12}^2-2d_1d_2 t_{11}t_{22}$
which is not a square in $L(\bft)$. Let $n\geq 3$.
Consider the ideals $\afrak_1 = (t_{1 2},\ldots, t_{1 n})$ and
$\afrak_2 = (t_{1n},\ldots, t_{(n-1)n})$ which correspond to the
respective substitutions $t_{1 2}=\cdots= t_{1 n} = 0$ and
$t_{1n}=\ldots=t_{(n-1)n} = 0$. Let $p_1,p_2$
be the reduction of $p$ modulo $\afrak_1,\afrak_2$
respectively.
As $TD \equiv \left(
\begin{array}{c|ccc}
d_1t_{11}  &       &   &  \\
\hline
 & d_2t_{22} &  \cdots       & d_nt_{2n}\\
 & \vdots      &    & \vdots  \\
 & d_2t_{2n} &  \cdots       & d_nt_{nn}
\end{array}\right) \mod \afrak_1$ the polynomial $p_1$ decomposes
as $p_1(\bft,x) = (x - d_1t_{11}) h(\bft,x)$ in $L[\bft,x]/\afrak_1
= L[t_{1,1},t_{ij},x\mid j\geq i\geq 2]$. By induction, $h$, which
is the characteristic polynomial of the lower block, is a separable
irreducible polynomial with Galois group $S_{n-1}$ over
$L(t_{i,j}\mid j\geq i\geq 2)$. In particular $p(\bft,x)$ is also
separable.

Now assume that $p(\bft,x)$ is reducible in $L[\bft,x]$, namely
\begin{equation*}\label{eq:decomposition}
p(\bft,x) = f(\bft,x) g(\bft,x),
\end{equation*}
where $f$ and $g$ are monic in $x$ and $1\leq \deg_x f \leq \deg_x
g$. Then the irreducibility of $h$ implies that $f \equiv x -
d_1t_{11}\pmod{\afrak_1}$. A similar argument (for
$p\mod{\afrak_2}$) implies $f \equiv x - d_nt_{nn}\pmod{\afrak_2}$.
Consequently, we get that
\[
f(\bft,x) - (x - d_1t_{11}) \in \afrak_1 \quad \mbox{and}\quad
f(\bft,x) - (x - d_nt_{nn}) \in \afrak_2 ,
\]
hence $d_1t_{11} - d_nt_{nn}\in \afrak_1 + \afrak_2$, a
contradiction.

Finally, we calculate the Galois group $G$ of $p(\bft,x)$ over
$L(\bft)$: As $p_1 = (x-d_{1}t_{11}) h(\bft,x)$ in
$L[\bft,x]/\afrak_1$ and the polynomials are separable, there exists
a bijection between the roots of $p$ and the roots of $p_1$. Such a
bijection induces an embedding of the Galois group of $h$ over
$L(t_{ij}\mid j\geq i\geq 2)$ into $G$ via the action on the
respective roots (c.f., \cite[Lemma 6.1.4]{FriedJarden2005}). That
is $S_{n-1}\leq G\leq S_n$ under a suitable labeling of the roots.
As $p(\bft,x)$ is separable and irreducible, $G$ is transitive, and
hence $G=S_n$. Indeed, for any $\sigma\in S_n$ there exists $\tau\in
G$ such that $\sigma(n)=\tau(n)$, so $\sigma \in \tau
S_{n-1}\subseteq G$.
\end{proof}

The next lemma is a weak Hilbert's Irreducibility Theorem for PAC extensions.
\begin{lemma}\label{prop:weakHIT}
Let $M/K$ be a PAC extension, let $f(t_1,\ldots, t_s,x)\in
M[t_1,\ldots, t_s,x]$ be a separable polynomial of degree $n$ in
$x$, and let $\tilde M$ be an algebraic closure of $M$. Assume that
the Galois group of $f(t_1,\ldots, t_s,x)$ over $\tilde
M(t_1,\ldots, t_s)$ is the symmetric group $S_n$. Then there exist
infinitely many $(\alpha_1,\ldots,\alpha_s)\in K^s$ such that
$f(\alpha_1,\ldots, \alpha_s,x)$ is irreducible over $M$, provided
that $M$ has a separable extension of degree $n$.
\end{lemma}

The proof appears in \cite[Cor.
2]{Bary-SorokerDTRF} (the proof is given for $s=1$, but it is easy to check that the same proof works for $s>1$).

\begin{proof}[Proof of Theorem~\ref{thm:main}]
Let $M/K$ be a PAC extension of characteristic $\neq 2$ and let $Q$
be a non-degenerate quadratic form over $K$. Choose a basis in which
$Q$ is diagonal and denote by
$D$ 
the corresponding diagonal matrix. By Lemma \ref{lem:Her}, it
suffices to find a symmetric matrix $A$ (with coefficients in $K$)
such that $AD$ has an irreducible characteristic polynomial. For
that, take $T$ to be the generic symmetric matrix with indeterminate
coefficients. By Lemma~\ref{lem:GaloisGroup} (with $L=\tilde M$) the
characteristic polynomial, $p(\bft,x)$, of $T D$ satisfies the
condition of Lemma~\ref{prop:weakHIT}. Thus if $M$ has a separable
extension of degree $n$, we can specialize $\bft \mapsto \bfa\in
K^{n^2}$ such that $p(\bfa,x)$ is irreducible. We thus get, for the
specialized matrix $A = (a_{ij})$, that $AD$ has an irreducible
characteristic polynomial.
\end{proof}

\section{Fields with PAC extensions}\label{sec:3}
In order to deduce Theorems \ref{thm:sol} and \ref{thm:pi} from
Theorem \ref{thm:main}, we need to show that the fields in question
have a PAC extension having suitable separable extensions. We start
by citing some results regarding PAC extensions that we will need.

Let $K_0$ be a field and $e\geq 1$ an integer. For
$\bfsigma=(\sigma_1,\ldots,\sigma_e)\in \gal(K_0)^e$, we denote by
$\left<\bfsigma\right> = \left<\sigma_1,\ldots,\sigma_e\right>$ the
closed subgroup of $\gal(K_0)$ generated by the coordinates of
$\bfsigma$, and by $K_{0s}(\bfsigma)$ the fixed field of
$\left<\bfsigma\right>$ in a fixed separable closure $K_{0s}$ of
$K_{0}$. We recall that the absolute Galois group of $K_{0}$ is
profinite (in particular compact), and hence equipped with a
probability Haar measure.

\begin{proposition}\label{thm:cite}
Let $K_0$ be a countable Hilbertian field and $K/K_0$ an algebraic
extension.
\renewcommand{\labelenumi}{\alph{enumi}.}
\begin{enumerate}
\item
For almost all $\bfsigma\in \gal(K_{0})^e$ the extension
$K_{0s}(\bfsigma)/K_0$ is a PAC extension  and $\left<
\bfsigma\right>$ is a free profinite group of rank $e$
({\cite[Proposition 3.1]{JardenRazon1994}} and \cite[Theorem
18.5.6]{FriedJarden2005}).
\item
If $M_0/K_0$ is a PAC extension, then so is $M_0K/K$
({\cite[Corollary 2.5]{JardenRazon1994}}).
\end{enumerate}
\end{proposition}

We shall also use the following simple group theoretic lemma.

\begin{lemma}\label{lem:grp}
Let $N\leq N_0\leq G$ be profinite groups such that $N$ is normal in
$G$. Let $H$ be a quotient of $G$ such that $H$ and $G/N$ have no
common nontrivial quotients. Then, $H$ is a quotient of $N_0$. In
particular, if $H$ has an open subgroup of index $n$, so does $N_0$.
\end{lemma}

\begin{proof}
Let $U\lhd G$ such that $G/U= H$. Since $G/NU$ is a common quotient
of $G/U$ and $G/N$, we get that $G/NU=1$, so $G=NU$. Therefore
 also $N_0U=G$, and hence $N_0/N_0\cap U \cong N_0U/U = G/U = H$.
\end{proof}

\subsection{Prosolvable extensions}
The following proposition shows that prosolvable extensions of a
countable Hilbertian field have many PAC extensions (cf.
\cite[Corollary 3.7]{Bary-SorokerDTRF}).
\begin{proposition}\label{prop:sol}
Let $K$ be a prosolvable extension of a countable Hilbertian field
$K_0$ and $e\geq 2$. Then for almost all $\bfsigma\in \gal(K_0)^e$
the field $M=KK_{0s}(\bfsigma)$ is a PAC extension of $K$ and it has
a separable extension of every degree $>4$.
\end{proposition}

\begin{proof}
Let $\hat K$ be the Galois closure of $K/K_0$, so $\gal(\hat K/K_0)$
is prosolvable. For almost all $\bfsigma\in \gal(K_0)^e$ the field
$M_0=K_{0s}(\sigma)$ is a PAC extension of $K_0$ and its absolute
Galois group $G=\left< \bfsigma \right>$ is a free profinite group
of rank $e$ (Proposition~\ref{thm:cite}a.). Fix such a $\bfsigma$
and write $M = KK_{0s}(\bfsigma)$. Then $M/K$ is PAC
(Proposition~\ref{thm:cite}b.). Let $N_0 = \gal(M)$ be the absolute
Galois group of $M$ and let $N = \gal(\hat KK_{0s}(\bfsigma))$. Then
$N\leq N_0\leq G$, $N$ is normal in $G$, and $G/N = \gal(\hat K
K_{0s}(\bfsigma)/K_0)$. The restriction map $G/N \to \gal(\hat
K/K_0)$ is an embedding, so $G/N$ is prosolvable.

Let $n>4$, we show that $M$ has a separable extension of degree $n$.
By Galois correspondence, it suffices to show that $N_0$ has an open
subgroup of index $n$. As $G$ is free of rank $\geq 2$ it has $A_n$
(the alternating group) as a quotient. $A_n$ and $G/N$ have no
nontrivial common quotients (as $G/N$ is prosolvable and $A_n$ is
simple). Now Lemma \ref{lem:grp} with $H=A_n$ implies that $N_0$ has
an open subgroup of index $n$ (since $(A_n:A_{n-1})=n$).
\end{proof}

\begin{proof}[Proof of Theorem \ref{thm:sol}]
Let $K$ be a prosolvable extension of a Hilbertian field $K_0$. In
case $K$ is countable, the assertion follows immediately from
Theorem~\ref{thm:main} and the above proposition. When $K$ is
uncountable, the assertion follows from the countable case by the
L\"owenheim-Skolem Theorem \cite[Proposition
7.4.2]{FriedJarden2005}.

Indeed, let $\mathcal L$ be the language of Rings together with a
unary predicate $P$.
By L\"owenheim-Skolem, there exists a countable elementary
substructure $E/E_0$ of the structure $K/K_0$, in particular,
$E/E_0$ is a field extension. We show that $E_0$ is Hilbertian and
that $E/E_0$ is prosolvable. This would imply that every
non-degenerate quadratic form over $E$ is isomorphic to a scaled
trace form. Since, for any fixed positive integer $n$, the statement
``All quadratic forms of dimension $n$ are isomorphic to scaled
trace forms" is elementary (by Lemma~\ref{lem:Her}), this would
conclude the proof.

To show that $E_0$ is Hilbertian, we need to show that for every
positive integer $n$, every irreducible polynomial $f(T,X)$ of
degree $n$ has an irreducible specialization. For any fixed $n$ this
is an elementary statement which is true in $K_0$ and hence also in
$E_0$. Next, we show that $E/E_0$ is separable algebraic and that
the Galois closure $\hat{E}$ of $E/E_0$ is linearly disjoint from
$K_0$ over $E_0$. Indeed, if $x\in E$, then $x\in K$. Hence there
exists an irreducible separable polynomial $f$ over $K_0$ satisfying
$f(x)=0$. The latter statement is elementary, so $x$ is separable
and algebraic over $E_0$. Now let $L$ be a finite separable
extension of $E_0$ and $f$ an irreducible generating polynomial of
$L/E_0$. Then $f$ generates $LK_0/K_0$ and is also irreducible over
$K_0$ (since it is elementary). Therefore $[L:E_0] = \deg f =
[LK_0:K_0]$, i.e., $L$ is linearly disjoint from $K_0$ over $E_0$.
As $L$ is an arbitrary finite separable extension, we get that
$E_{0s}$ is linearly disjoint from $K_0$ over $E_0$, and in
particular so is $\hat E$. Finally, let $\hat{K}$ be the Galois
closure of $K/K_0$. Then $\hat E K_0 \subseteq \hat K$, and
$\gal(\hat E/ E_0) \cong \gal(\hat E K_0/K_0)$ via restriction. As
$\gal(\hat K/K_0)$ is prosolvable, so is  $\gal(\hat E K_0/K_0)
\cong \gal(\hat K/K_0) / \gal(\hat K/\hat E K_0)$, and hence
$\gal(\hat E/ E_0)$ is prosolvable.
\end{proof}

\begin{remark}
Theorem~\ref{thm:sol} is valid in particular for the maximal
prosolvable extension $\mathbb{Q}_{\mathrm{sol}}$ of $\mathbb Q$.
However, $\mathbb{Q}_{\mathrm{sol}}$ has no quadratic extensions,
hence there is a unique non-degenerate quadratic form of a given
dimension (up to isomorphism). So in this maximal case the theorem
is obvious.
\end{remark}

Theorem~\ref{thm:sol} answers Question~\ref{que:main} positively for
solvable extensions and quadratic forms of dimension $n>4$. The
following proposition shows that for $n=2$ the answer is negative.
We do not know what happens in the gap $n=3,4$.

\begin{proposition} \label{prop:prosolvable}
There exists a prosolvable extension $K/\mathbb Q$ which has an
extension of degree $2$, but the non-degenerate quadratic form
$\langle 1,1\rangle$  is not isomorphic to a scaled trace form over
$K$.
\end{proposition}

\begin{proof}
Fix an embedding of $\mathbb Q$ in $\mathbb C$. Complex conjugation
acts nontrivially on $\mathbb{Q}_{\mathrm{sol}}$. Let $K$ be its
fixed field, i.e., $K = \mathbb{Q}_{\mathrm{sol}} \cap \mathbb R$.
Then $[\mathbb{Q}_{\mathrm{sol}}:K]=2$. Also if $L/K$ is an
extension of degree $2$, then $L$ is also a prosolvable extension of
$\mathbb Q$, so $L \subseteq \mathbb{Q}_{\mathrm{sol}}$, which
implies that $L=\mathbb{Q}_{\mathrm{sol}}$.  Thus
$\mathbb{Q}_{\mathrm{sol}}$ is the unique extension of degree $2$
over $K$. We now proceed as in $\mathbb C/\mathbb R$. Any scaled
trace form $x\mapsto \Tr_{\mathbb{Q}_{\mathrm{sol}}/K}(\alpha x^2)$
is isotropic (since $x=\sqrt{i/\alpha}\in
\mathbb{Q}_{\mathrm{sol}}$). The assertion follows since $\langle
1,1\rangle$ is not isotropic over $K$ (recall that $K$ is real).
\underline{}\end{proof}

\subsection{Prime-to-$p$ extensions}
Let $p, m,k$ be positive integers such that $p$ is prime, $p\nmid
m$, and $p\mid \phi(m)$. Here $\phi$ is Euler's totient function.
Consider the semidirect product $H = \bbZ/m\bbZ \rtimes \bbZ/p^k
\bbZ$ of all pairs $(a,x)$, $a\in \bbZ/m\bbZ$ and $x\in \bbZ/p^k
\bbZ$ with multiplication given by
\[
(a,x)(b,y) = (a + \alpha^x b, x+y),
\]
where $\alpha \in (\bbZ/m\bbZ)^*$ is a fixed element of
(multiplicative) order $p$. In particular
\[
(a,x)^n =(a(1+\alpha+\alpha^2 + \cdots +\alpha^{n}),nx) =
\left(\frac{a(1-\alpha^{n+1})}{1-\alpha},nx\right).
\]
We embed $\bbZ/m \bbZ$ and $\bbZ/p^k \bbZ$ in $H$ in the natural
way.

\begin{lemma}
\renewcommand{\labelenumi}{\alph{enumi}.}
Let $p,m,k$, and $H = \bbZ/m\bbZ \rtimes \bbZ/p^k \bbZ$ be as above.
Then
\begin{enumerate}\label{lem:semi}
\item
$H$ is generated by its $p$-sylow subgroups.
\item
The only prime-to-$p$ quotient of $H$ is trivial.
\item
If $n\mid m$, then there exist subgroups $H_0$ and $H_1$ of $H$ of
respective indices $p^k n$ and $n$.
\end{enumerate}
\end{lemma}

\begin{proof}
The elements $(0,1)$ and $(1,1)$ generate $H$, so for a., it
suffices to show that their order divides $p^k$ (and hence is
$p^k$). We have $(0,1)^{p^k} = (0,p^k)=(0,0)$. Now, since
$\alpha^{p^k}=(\alpha^p)^k=1$, we have
\[
(1,1)^{p^k} = \left(\frac{1-\alpha^{p^k}}{1-\alpha},0\right)=(0,0).
\]

b.\! follows from a.: Indeed, let $\bar H = H/N$ be a quotient of
$H$ with order prime-to-$p$. Thus $p^k$ divides the order of $N$,
and hence $N$ contains a $p$-sylow subgroup of $H$. As $N\lhd H$, it
contains all the $p$-sylow subgroups. By a., $N= H$ and $\bar H =
1$, as desired.

To show c., assume $n\mid m$. Let $H_0$ be the kernel of the natural
map $\bbZ/m\bbZ\to \bbZ/n\bbZ$. Then $H_0 \leq \bbZ / m \bbZ\leq H$.
Note that $H_0$ is invariant under the action of $\bbZ/p^k\bbZ$
(i.e., $a\in H_0\Rightarrow \alpha a\in H_0$) and define
$H_1=H_0\rtimes \bbZ/p^k\bbZ$. Now,
\[
(H:H_0) = (H:\bbZ/m\bbZ)(\bbZ/m\bbZ:H_0) = p^k n
\]
and $(H:H_1)=n$, as desired.
\end{proof}

%

\begin{proposition}
Let $p$ be a prime, let $K$ be a separable extension of a countable
Hilbertian field $K_0$ such that the Galois closure is prime-to-$p$
over $K_0$. Let $e\geq 2$. Then for almost all $\bfsigma\in
\gal(K_0)^e$ the field $M=KK_{0s}(\bfsigma)$ is a PAC extension of
$K$ and it has a separable extension of every degree.
\end{proposition}

\begin{proof}
Let $\hat K$ be the Galois closure of $K/K_0$, so every finite
quotient of $\gal(\hat K/K_0)$ has order prime to $p$. As in the
proof of Proposition \ref{prop:sol}, for almost all $\bfsigma\in
\gal(K_0)^e$ the field $M_0=K_{0s}(\bfsigma)$ has a free absolute
Galois group of rank $e$, namely $G=\left< \bfsigma \right>$, and
$M=M_0K$ is a PAC extension of $K$. Let $N_0 = \gal(M)$ be the
absolute Galois group of $M$ and let $N = \gal(\hat
KK_{0s}(\bfsigma))$. Then $N\leq N_0\leq G$, $N$ is normal in $G$,
and $G/N = \gal(\hat K K_{0s}(\bfsigma)/K_0)$. The restriction map
$G/N \to \gal(\hat K/K_0)$ is an embedding, so every finite quotient
of $G/N$ has order prime to $p$ (because it is a subgroup of a
finite quotient of $\gal(\hat K/K_0)$).

By Galois correspondence, it suffices to show that $N_0$ has open
subgroups of any index. Let $n$ be a positive integer prime to $p$
and $k\geq 1$. Let $l$ be a prime number such that $l\nmid n$ and
$p\mid l-1$ and let $m=nl$ (such a prime $l$ exists since there are
infinitely many primes in the arithmetic progression $l\equiv 1\pmod
p$). Then $p\nmid m$ and $p\mid \phi(m)$. Let $H = \bbZ/m\bbZ\rtimes
\bbZ/p^k\bbZ$ as in Lemma~\ref{lem:semi}. Then $H$ and $G/N$ have no
nontrivial common quotients (by Lemma~\ref{lem:semi}b.). By Lemma
\ref{lem:grp} and Lemma~\ref{lem:semi}c., $N_0$ has open subgroups
of index $n$ and $np^k$, i.e., of any index.
\end{proof}

\begin{proof}[Proof of Theorem \ref{thm:pi}]
For a countable field $K$, Theorem~\ref{thm:pi} follows immediately
from Theorem~\ref{thm:main} and the above proposition. For an
uncountable field the theorem follows from the L\"owenheim-Skolem
Theorem (as in the proof of Theorem \ref{thm:sol}).
\end{proof}

\subsection{A remark}
In this work we saw that the property for a field $K$ to have PAC
extension with separable extension of degree $n$, implies that every
non-degenerate quadratic form of dimension $n$ is isomorphic to a
scaled trace form. This leads to the problem of classifying fields
$K$ which have a PAC extension with a separable extension of a given
degree. This is a refinement of \cite[Problem
3.9]{Bary-SorokerDTRF}, raised in relation to Dirichlet's theorem
for polynomial rings in one variable.


\bibliographystyle{amsplain}

\end{document}